\documentclass[11pt]{amsart}

\usepackage[utf8]{inputenc} 

\usepackage{amscd,amsmath,amssymb,amsfonts,amsthm,mathrsfs}
\usepackage[a4paper,text={160mm,240mm},centering,headsep=5mm,footskip=10mm]{geometry}

\usepackage[OT2,T1]{fontenc}
\DeclareSymbolFont{cyrletters}{OT2}{wncyr}{m}{n}
\DeclareMathSymbol{\Sha}{\mathalpha}{cyrletters}{"58}

\usepackage[colorlinks,linkcolor=black,anchorcolor=black,citecolor=black]{hyperref}

\usepackage{tikz-cd}
\usepackage[all,cmtip]{xy}

\usepackage{array}
\usepackage{color}

\numberwithin{equation}{section}
\newtheorem{theorem}{Theorem}[section]
\newtheorem{lemma}[theorem]{Lemma}

\newtheorem{proposition}[theorem]{Proposition}
\newtheorem{corollary}[theorem]{Corollary}

\newtheorem{definition}[theorem]{Definition}
\newtheorem{remark}[theorem]{Remark}

\newtheorem*{notat*}{Notation}

\newtheorem{conj}[theorem]{Conjecture}
\newtheorem{quest}[theorem]{Question}
\newtheorem{condition}[theorem]{Condition}

\newcommand{\N}{\mathbb{N}}
\newcommand{\Z}{\mathbb{Z}}
\newcommand{\Q}{\mathbb{Q}}

\renewcommand{\O}{\mathcal{O}}

\renewcommand{\r}{\rightarrow}

\newcommand{\Nis}{{\rm Nis}}
\newcommand{\et}{\mathrm{\acute{e}t}}
\newcommand{\Spec}{\mathrm{Spec}}

\newcommand{\CH}{\mathrm{CH}}
\newcommand{\Br}{{\rm Br}}
\renewcommand{\ker}{\mathrm{ker}}

\newcommand{\resprod}{\sideset{}{^{\prime}}\prod}

\begin{document}
\title[A local to global principle]{A local to global principle for higher zero-cycles}
\author{Johann Haas and Morten Lüders}
\address{Fakult\"at f\"ur Mathematik, Universit\"at Regensburg, 93040 Regensburg, Germany}
\email{johann.haas@ur.de}
\address{Institut de Mathématiques de Jussieu–Paris Rive Gauche, UMPC - 4 place Jussieu, Case 247, 75252 Paris, France}
\email{morten.lueders@ur.de}
\date{\today}
\thanks{The first author is supported by the DFG through CRC 1085 \textit{Higher Invariants} (Universit\"at Regensburg). The second author is supported by a DFG research fellowship.}
%\subjclass[2010]{Primary ; Secondary .}

\begin{abstract}
We study a local to global principle for certain higher zero-cycles over global fields. We thereby verify a conjecture of Colliot-Th\'el\`ene for these cycles. Our main tool are the Kato conjectures proved by Jannsen, Kerz and Saito. Our approach also allows to reprove the ramified global class field theory of Kato and Saito.
Finally, we apply the Kato conjectures to study the $p$-adic cycle class map over henselian discrete valuation rings of mixed characteristic and to deduce finiteness theorems for arithmetic schemes in low degree. 
\end{abstract}
\maketitle
\tableofcontents

\setcounter{page}{1}
\setcounter{section}{0}
\pagenumbering{arabic}

\section{Introduction}
Let $K$ be a number field, $X$ a smooth projective geometrically integral variety over $K$ and $n$ a positive integer. Let $P_K$ denote the set of places of $K$. The following conjecture was suggested by Kato and Saito in \cite[Sec. 7]{KaS86} and Colliot-Th\'el\`ene in  \cite[Conj. 1.5(c)]{Co95} (see also \cite{CS81} for the case of rational surfaces).
\begin{conj}\label{conjA}
The complex
\begin{equation*}\label{conj1}
\varprojlim_nA_0(X)/n\r \prod_{v\in P_K}\varprojlim_nA_0(X_{K_v})/n\r \mathrm{Hom}(\Br(X)/\Br(k),\Q/\Z)
\end{equation*}
is exact.
\end{conj}

Conjecture \ref{conjA} is known to hold if $X$ is a curve and if the Tate-Shafarevich group of the Jacobian of $X$ does not contain a non-zero element which is infinitely divisible (see \cite[Sec. 7]{Sa89'}, \cite[Sec. 3]{Co99} and \cite[Rem. 1.1(iv)]{Wi12} as well as Remark \ref{remarkconj11forcurves}). For more details on this conjecture and results for higher dimensional schemes we refer to \cite{Wi12} and \cite{Wi16}. 

In \cite{Co99}, Colliot-Th\'el\`ene considers a slightly weaker conjecture on the image of the complex (\ref{conj1}) in \'etale cohomology. For our purposes we will generalize this conjecture to Bloch's higher Chow groups. In order to state the conjecture, we need to recall Saito's exact sequence generalizing the Tate-Poitou exact sequence (see Section \ref{sectiontatepoitou}) to schemes (\cite{Sa89}). For later purposes we state it in its most general form proven by Geisser and Schmidt (\cite{GS18}). Let $K$ be a global field. Let $S$ be a non-empty and possibly infinite set of prime divisors of $K$ containing the archemedian primes if $K$ is a number field. Let $\O_S$ be the ring of elements in $K$ which are integers at all primes $\mathfrak{p}\notin S$. Let $\mathcal{S}=\Spec\O_S$ and $\mathcal{X}$ a regular, flat and separated scheme of finite type of relative dimension $d$ over $\mathcal{S}$. Let $n$ be a positive integer invertible on $\mathcal{S}$ and let $\mathcal{F}$ be a locally constant constructible sheaf of $\Z/n\Z$-modules on $\mathcal{X}$. Let $D(\mathcal{F})=\mathrm{Hom}(\mathcal{F},\mu_n^{\otimes d+1})$ be its dual. For an abelian group $A$, let $A^\vee=\mathrm{Hom}(A,\Q/\Z)$. Then there is an exact sequence
\begin{equation}\label{SaitoTatePoitou}
\ldots\r H^i_{\et}(\mathcal{X},\mathcal{F})\r \resprod_{v\in S}\tilde{H}^{i}_{\et}(X_{K_v},\mathcal{F})\r H^{2d+2-i}_{\et,c}(\mathcal{X}_{},D(\mathcal{F}))^\vee \r H^{i+1}_{\et}(\mathcal{X},\mathcal{F})\r\ldots
\end{equation} 
Here the groups  $\tilde{H}^{i}_{\et}(X_{K_v},\mathcal{F})$ are the modified cohomology groups defined in \cite[Def. 1.6]{Sa89}, resp. \cite[Sec. 2]{GS18}, and in the product we restrict to classes which are almost everywhere unramified. In Definition~\ref{definition:restricted_product_of_Chow_groups} we define an analogous restricted product of higher Chow groups $\prod_{v\in S}' \CH^i(X_{K_v}a,\Z/n\Z)$. Let now $S=P_K$, $\mathcal{X}=X$ be smooth projective over $\Spec K$ and $\mathcal{F}=\mu_n^{\otimes i}$.
By construction, cf. Remark~\ref{remark:properties_of_restricted_product_of_chow_groups}, we then get a commutative diagram
$$\begin{xy} 
  \xymatrix{
  \CH^i(X,a,\Z/n\Z) \ar[r] \ar[d]_{} & \prod_{v\in P_K}' \CH^i(X_{K_v}a,\Z/n\Z) \ar[d]_{}  &  \\
H^{2i-a}_{\et}(X,\mu_n^{\otimes i}) \ar[r]  & \prod_{v\in P_K}'\tilde{H}^{2i-a}_{\et}(X_{K_v},\mu_n^{\otimes i}) \ar[r]^{} & H^{2d+2-2i+a}_{\et}(X,\mu_n^{\otimes d+1-i})^\vee
  }
\end{xy} $$
and in particular a pairing 
$$\resprod_{v\in P_K} \CH^i(X_{K_v}a,\Z/n\Z) \times H^{2d+2-2i+a}_{\et}(X_{},\mu_n^{\otimes d+1-i})\r \Q/\Z$$
which is zero on $  \CH^i(X,a,\Z/n\Z)$. The following conjecture is a version of the mentioned weaker conjecture of Coliot-Th\'el\`ene extended to Bloch's higher Chow groups:

\begin{conj}(see \cite[Conj. 2]{Co99})\label{conj2}
Let $X/K$ be a smooth projective, geometrically integral, variety over a global field $K$ and $n$ be prime to $\mathrm{K}$. Let $d$ be the dimension of $X$, and $i>0$ an integer. Let $\{z_v\}_{v\in P_K}\in \prod_{v\in P_K}' \CH^i(X_{K_v},a,\Z/n\Z)$ and suppose that every class $\xi\in H^{2d+2-2i+a}_{\et}(X_{},\mu_n^{\otimes d+1-i})$ is orthogonal to the family $\{z_v\}_{v\in P_K}$. Then there exists a global cycle $z_n\in \CH^i(X,a)$ such that for every place $v$ the class of $z_n$ in  $\tilde{H}^{2i-a}_{\et}(X_{K_v},\mu_n^{\otimes i})$ coincides with that of $z_v$.
\end{conj}

Furthermore, we can state the following conjecture analogous to Conjecture \ref{conj1}:
\begin{conj}\label{conj3}
Let $j=d+1-i$. The complex
\begin{equation*}
\CH^{i}(X,a,\Z/n\Z)\r \resprod_{v\in P_K} \CH^{i}(X_{K_v}a,\Z/n\Z)\r H^{2j+a}_{\et}(X_{},\mu_n^{\otimes j})^\vee
\end{equation*}
is exact.
\end{conj}

Let $\mathcal{X}$ and $S$ be as above. In this case we define Bloch's higher Chow groups by taking the Zarski-hypercohomology of Bloch's cycle complex $z^n(-,i)$ on $\mathcal{X}$, see \cite[Cap. 3]{Ge04}. For $\mathcal{S}$ a semi-local Dedekind domain this coincides with taking the homology of Bloch's cycle complex due to the existence of a localisation sequence. For $\mathcal{S}$ a field this is due to Bloch and for $\mathcal{S}$ semi-local this is shown by Levine in \cite{Le01} (see also \cite{Ge04}). We introduce the following notation:
$$\Sha(\CH^{d+1}(\mathcal{X},a,\Z/n\Z)):=\mathrm{ker}[\CH^{d+1}(\mathcal{X},a,\Z/n\Z)\r \prod_{v\in S}\CH^{d+1}(X_{K_v},a,\Z/n\Z)]$$
and $$\Sha^{2d+2-a,d-1}(\mathcal{X}):=\mathrm{ker}[H^{2d+2-a}_{\et}(\mathcal{X},\Z/n\Z(d+1)) \r \prod_{v\in S} H^{2d+2-a}_{\et}(X_{K_v},\Z/n\Z(d+1))]$$  
These groups define analogues of the Tate-Shavarevich group for Chow groups and \'etale cohomology. 

Throughout the article we will often need the following condition:
\begin{condition}
	Let $K_v$ be a non-archimedean local field, $V / K_v$ a smooth scheme.
	\begin{itemize}
		\item[($\star$)] $V$ is projective and admits a model over $\O_{K_v}$ whose special fiber is smooth except at finitely many points, where it has ordinary quadratic singularites.
		\item[($\star\star$)] $V$ is projective and has strictly semistable reduction.
	\end{itemize}
\end{condition}

We are now able to state our main theorem:

\begin{theorem}[Thm. \ref{maintheorem2intext}]\label{maintheorem2}
Let $K$ be a global field, $S$ a set of places of $K$ and $n \in \N_{>1}$. Denote by $\O_S$ the ring of elements in $K$ which are integers at all primes $\mathfrak{p}\notin S$ and let $\mathcal{S}=\Spec \O_S$.
Suppose that
\begin{enumerate}
	\item[(a)] $\mathcal{S}$ is either semi-local or an open of $\Spec(\O_K)$.
	\item[(b)] $n$ is invertible on $\mathcal{S}$.
	\item[(c)] If $K$ is a number field, $S$ contains all archimedian places of $K$ and either $n$ is odd or $K$ has no real places. 
\end{enumerate} 
Let $\mathcal{X}$ be regular, flat and projective of relative dimension $d$ over $\mathcal{S}$ with smooth generic fiber $X$. 
Then the following statements hold:
\begin{enumerate}
\item There is an exact sequence 
$$\CH^{d+1}(\mathcal{X},a,\Z/n\Z)\r  \resprod_{v\in S}\CH^{d+1}(X_{K_v},a,\Z/n\Z) \r H^{a}_{\et}(\mathcal{X}_{},\Z/n\Z)^\vee$$
\item For all $a$ there is a natural surjection
$\Sha(\CH^{d+1}(\mathcal{X},a,\Z/n\Z))\r \Sha^{2d+2-a,d+1}_\et(\mathcal{X}).$
\item If $K$ is a number field and condition ($\star$) holds for $X_{K_v}$ if $v$ divides $n$, then the group $\Sha(\CH^{d+1}(\mathcal{X},1,\Z/n\Z))$ is finite.
\item If $K$ is a function field of one variable over a finite field and $n$ is invertible in $K$, then $\Sha(\CH^{d+1}(\mathcal{X},a,\Z/n\Z))$ is finite for arbitrary $a$.
\end{enumerate}
\end{theorem}

Note that $\CH^{d+1}(X,a)=0$ for $a\leq 0$ by dimension reasons and that 
$\CH^s(X,2s-r)\cong H^r(X,\Z(s))$
for $X$ smooth over a perfect field. Theorem \ref{maintheorem2}(2) is related to a conjecture of Bloch (see \cite[Conj. 3.16]{Bl91}). It means that the Tate-Shafarevich group $\Sha^{2d+2-a,d-1}_\et(X)$ is - under the above assumptions - generated by algebraic cycles.

\begin{corollary}[Cor. \ref{corollarymaintheorem2}]\label{Cormaintheorem} Let the assumptions be as in Theorem \ref{maintheorem2}. Let $S=P_K$ and $X=\mathcal{X}$. Then the following statements hold:
\begin{enumerate}
\item Conjecture \ref{conj2} holds for $i \geq d+1$ and all $a$.
\item Conjecture \ref{conj3} holds for $i \geq d+1$ and all $a$.
\end{enumerate}
\end{corollary}

If $a=1$, $S=P_K$ and $X=\mathcal{X}$, then Theorem \ref{maintheorem2} recovers the unramified class field theory of Bloch, Kato and Saito (see \cite{Bl81}, \cite[Thm. 3]{KaS83}, \cite[Sec. 5]{Sa85} and \cite[Thm. 2.10]{Ka86}). For this note that 
$$\CH^{d+1}(X,1)\cong \text{coker}[\bigoplus_{x\in X^{(d-1)}}K_{2}^Mk(x)\rightarrow \bigoplus_{x\in X^{(d)}}K_{1}^Mk(x)]=:SK_1(X)$$
(see e.g. \cite{Lu17}). In Section \ref{sectionraamifiedcft} we extend this corollary to cover the ramified global class field theory of Kato and Saito. For this we use our method of Section \ref{secmainthm} and an idea of Kerz and Zhao in their approach to class field theory over local and finite fields.
\begin{theorem}(\ref{theoremunramified})
Let $K$ be a global field. Let $n$ be invertible in $K$. If $K$ is a number field assume furthermore that either $n$ is odd or that $K$ has no real places. Let $X$ be a smooth projective scheme over $\Spec(K)$. Let $D\subset X$ be an effective divisor on $X$ and $j:U\hookrightarrow X$ be the inclusion.
Then there is an isomorphism
$$\mathrm{coker}[H^d_\Nis(X,\mathcal{K}^M_{d+1,X|D}/n)\r \resprod_{v\in P_K}H^d_\Nis(X_{K_v},\mathcal{K}^M_{d+1,X_{K_v}|D_{K_v}}/n)]\r \pi_1^{\mathrm{ab}}(U)/n.$$
\end{theorem}

In Section \ref{sectionpadic} we show the following theorem:

\begin{theorem}[Thm. \ref{thm3intext}]\label{thm3}
Let $A$ be a henselian discrete valuation ring of characteristic zero with residue field of characteristic $p$ and function field $K$. Let $X$ be smooth and projective over $\Spec(A)$. Let $X_1$ denote the special and $X_K$ the generic fiber of $X$. Let $d=\mathrm{dim}X-1\leq 2$. Then there is an isomorphism
$$\CH^{d+1}(X,1,\Z/p^r\Z)\r H^{2d+1}_\et(X,\mathcal{T}_r(d+1)),$$
where $\mathcal{T}_r(d+1)$ are the $p$-adic \'etale tate twists defined in \cite{Sa07}.
\end{theorem}

The following corollary follows from \cite[Prop. 1.4]{Lu17'} and answers a question posed in \textit{loc.~cit.}.
\begin{corollary}
Let the situation be as in Theorem \ref{thm3}. Assume that $A=W(k)$ for a finite field $k$ of characteristic $p>d+2$. Let $X_n$ denote the thickenings of the special fibre $X_1$ and $\mathcal{K}^M_{*,X_n}$ be the improved Milnor K-sheaf on $X_n$ defined in \cite{Ke10}. Then there is an isomorphism of pro-abelian groups
$$\CH^{d+1}(X,1,\Z/p^r\Z)\r "\mathrm{lim}_n" H^{d}(X_1,\mathcal{K}^M_{d+1,X_n}/p^r).$$
\end{corollary}

In Section \ref{sectionfiniteness} we show a finiteness theorem in low degree for higher Chow groups of schemes over Dedekind domains. The idea of the proof can also be found in \cite[Sec. 7.2]{Ge10}.

\section{Restricted products of Chow groups}
We introduce restricted products of Chow groups over global fields. In the case of $i=d+1$ and $a=1$ these were studied in arithmetic class field theory by Bloch, Kato and Saito.
\begin{definition}\label{definition:restricted_product_of_Chow_groups}
	Let $K$ be a global field, $X/K$ a smooth projective geometrically integral variety, $n$ a positive integer prime to $\mathrm{ch}(K)$ and $i,a \in \N$. Let $\O_K$ be the ring of integers of $K$, and $\mathcal{S} \subset \Spec(\O_K)$ a dense open over which $X$ has a smooth projective model $\mathcal{X}$. Define
	\[
	\resprod_{v \in P_K} \CH^i(X_{K_v},a,\Z/n\Z) \subset \prod_{v \in P_K} \CH^i(X_{K_v}a,\Z/n\Z)
	\] 
	to be the subgroup of classes $(\alpha_v)_{v \in P_K}$ satisfying the condition that
	\[
	\alpha_v \in \mathrm{im}(\CH^i(\mathcal{X}_{\mathcal{O}_{K_v}},a,\Z/n\Z) \to \CH^i(X_{K_v},a,\Z/n\Z))
	\]
	for almost all $v \in \mathcal{S}$.
\end{definition}
\begin{remark}\label{remark:properties_of_restricted_product_of_chow_groups}~
\begin{enumerate}
	\item This definition does not depend on the choice of $\mathcal{S}$ nor on that of $\mathcal{X}$ by standard spreading-out arguments.
	\item For $a = 0$, the restricted product of Chow groups agrees with the usual one, since the restriction map $\CH^i(\mathcal{X}_{\mathcal{O}_{K_v}})\r \CH^i(X_{K_v})$ is surjective. This does not hold for higher Chow groups, which is why we need to introduce the restricted product of Chow groups.
	\item Restricted products of \'etale cohomology groups are defined in an analogous manner. In particular, the \'etale cycle class map restricts to a morphism
	\[
	\resprod_{v\in P_K} \CH^i(X_{K_v},a,\Z/n\Z)\r \resprod_{v\in P_K}\tilde{H}^{2i-a}_{\et}(X_{K_v},\mu_n^{\otimes i}).
	\]
	\item By spreading out cycles, one sees that the pullback map
	\[
		\CH^i(X,a,\Z/n\Z) \to \prod_{v \in P_K} \CH^i(X_{K_v},a,\Z/n\Z)
	\]
	factors over the inclusion of the restricted product.
\end{enumerate}
\end{remark}
In certain degrees, one can detect the restriced product fully on the level of \'etale cohomology:
\begin{lemma}\label{lemma:pullback_square_restricted}
	For $i \geq d+1$, the diagram
	\[
	\begin{tikzcd}
		\resprod_{v\in P_K} \CH^i(X_{K_v},a,\Z/n\Z) \arrow[r]\arrow[d] & \resprod_{v\in P_K}\tilde{H}^{2i-a}_{\et}(X_{K_v},\mu_n^{\otimes i}) \arrow[d] \\
		\sideset{}{}\prod_{v\in P_K} \CH^i(X_{K_v},a,\Z/n\Z) \arrow[r] & \sideset{}{}\prod_{v\in P_K}\tilde{H}^{2i-a}_{\et}(X_{K_v},\mu_n^{\otimes i})
	\end{tikzcd}
	\]
	is a pullback square of abelian groups.
\end{lemma}
\begin{proof}
	Let $S \subset \Spec(\O_K)$ a dense open over which $X$ has a smooth projective model $\mathcal{X}$ and over which $n$ is invertible. For a place $v \in P_K$ lying in $S$, consider the commutative diagram of localization sequences
	\[
	\begin{tikzcd}
	\CH^i(\mathcal{X}_{\mathcal{O}_{K_v}},a,\Z/n\Z) \arrow[r]\arrow[d]& \CH^i(X_{K_v},a,\Z/n\Z) \arrow[r]\arrow[d]& \CH^{i-1}(\mathcal{X}_{k_v},a-1,\Z/n\Z)\arrow[d] \\
	{H}^{2i-a}_{\et}(\mathcal{X}_{\mathcal{O}_{K_v}},\mu_n^{\otimes i})\arrow[r]& {H}^{2i-a}_{\et}(X_{K_v},\mu_n^{\otimes i}) \arrow[r]& {H}^{2i-a-1}_{\et}(\mathcal{X}_{k_v},\mu_n^{\otimes i-1})
	\end{tikzcd}
	\]
	and note that the rightmost vertical arrow is an isomorphism by \cite[Theorem 9.3]{KeS12} for $i = d+1$ (see also \cite[Lemma 6.2]{JS}). This means that the natural map
	\[
		\CH^i(\mathcal{X}_{\mathcal{O}_{K_v}},a,\Z/n\Z) \to \CH^i(X_{K_v},a,\Z/n\Z) \times_{{H}^{2i-a}_{\et}(X_{K_v},\mu_n^{\otimes i})} {H}^{2i-a}_{\et}(\mathcal{X}_{\mathcal{O}_{K_v}},\mu_n^{\otimes i})
	\]
	is a surjection. Together with the fact that the inclusions maps from the restricted products are injective, this is enough to establish the claim.
\end{proof}

\section{The Tate-Poitou exact sequence}\label{sectiontatepoitou}
We recall results of Tate and Poitou (see \cite[Thm. 3.1]{Ta63}). These will be generalised in many directions in higher dimension by the Kato conjectures, higher dimensional class field theory and local to global principles for (higher) Chow groups. 

We introduce the following notation: let $K$ be a global field. Let $S$ be a non-empty and possibly infinite set of prime divisors of $K$ containing the archemedian primes if $K$ is a number field. Let $K_S$ be the ring of elements in $K$ which are integers at all primes $\mathfrak{p}\notin S$. Let $K_S$ denote the maximal extension of $K$ in $K^s$ that is ramified over $K$ only at primes in $S$. Let $G_S:=\mathrm{Gal}(K_S/K)$. Let $M$ be a finite $G_S$-module of order $n$ which is invertible on $K_S$. Let $D(M):=\mathrm{Hom}(M,\mathbb{G}_m)$. For an abelian group $A$ let $A^\vee:=\mathrm{Hom}(A,\Q/\Z)$.
\begin{theorem}\label{TatePoitou}
\begin{enumerate}
\item For $i\geq 3$ the natural map
$$H^i(K_S,M)\r \prod_{v\in S}H^i(K_v,M)$$ 
is an isomorphism.
\item There is an exact nine-term sequence
\begin{align*}
0\r  H^0(K_S,M)\r &\prod'_{v\in S}H^0(K_v,M)\r  H^2(K_S,D(M))^\vee  \r \\
 H^1(K_S,M)\xrightarrow{\alpha_1}  & \prod'_{v\in S}H^1(K_v,M)\xrightarrow{\beta_1}   H^1(K_S,D(M))^\vee  \r \\
 H^2(K_S,M)\xrightarrow{\alpha_2} & \prod'_{v\in S}H^2(K_v,M)\xrightarrow{\beta_2}   H^0(K_S,D(M))^\vee  \r 0 
\end{align*}
\end{enumerate}
\end{theorem}
Here $\prod'_{v\in S}$ denotes the restricted product with respect to the subgroups $H^i(\O_{K_v},M)$. At the archemedian places we assume the cohomology groups to be the completed cohomology groups (see \cite[Ch. VIII]{Se95}). Note that the restricted product in the first line becomes a direct product since $H^0(\O_{K_v},M)\r H^0(K_v,M)$ is surjective and that the restricted product in the third line becomes a direct sum since $H^2(\O_{K_v},M)=0$. 

For $M=\mu_n$ this sequence encodes fundamental theorems in algebraic number theory. Let $S=P_K$.
Since $H^1(K,\mu_n)\cong K^\times/n$ and $H^1(K,\Z/n\Z)\cong \mathrm{Gal}(K^{\mathrm{ab}}/K)/n$ one recovers the class field theory isomorphism
$$C_K/n:=\mathrm{coker}[K^\times/n\r \prod'_{v\in P}K_v^\times/n]\xrightarrow{\cong}\mathrm{Gal}(K^{\mathrm{ab}}/K)/n$$
from the second line since $\beta_1$ is surjective by the density of the Frobenii. Since $H^2(K,\mu_n)\cong \Br(K)[n]$, the third line recovers, after passing to the direct limit, the Brauer-Hasse-Noether exact sequence
$$0\r \Br(K)\r\bigoplus_{v\in P_K} \Br(k_v)\r \Q/\Z\r 0.$$

Noting that $\CH^a(K,a)\cong K^M_a(K)$ (see \cite[Thm. 6.1]{Bl86}, \cite[Thm. 4.9]{NS89} and \cite{To92}) and applying $\mathrm{Hom}(\_,\Q/\Z)$ to the second sequence, these may be interpreted as results about (higher) Chow groups. The second line then becomes 
$$C_K/n:=\mathrm{coker}[\CH^1(K,1,\Z/n\Z)\r \prod'_{v\in P}\CH^1(K_v,1,\Z/n\Z)]\xrightarrow{\cong}\mathrm{Gal}(K^{\mathrm{ab}}/K)/n$$
and the third line becomes
\begin{equation}\label{conj1forK}
\varprojlim_n\CH_0(K)/n\r \varprojlim_n\prod'_{v\in P}\CH_0(K_v)/n\r \mathrm{Hom}(\Br(K),\Q/\Z)
\end{equation}
which asserts Conjecture \ref{conj1} for $\Spec(K)$ (see also \cite[Rem. 1.1]{Wi12}). For the latter statement one may also consider the Tate-Poitou exact sequence for $M=\Z/n\Z$. In this case the first line becomes
$$\begin{xy} 
  \xymatrix{
	0\ar[r]&\Z/n\Z \ar[r]\ar[d]^{\cong}& \prod'_{v\in P_K} \Z/n\Z \ar[r]\ar[d]^{\cong}&  H^2(K,\mu_n)^\vee =(\Br(X)[n])^\vee \ar[d]^{=} \\
 0\ar[r]&	H^0(K,\Z/n) \ar[r]& \prod'_{v\in P_K}H^0(K_v,\Z/n\Z)\ar[r] & H^2(K,\mu_n)^\vee}
	\end{xy}
$$	
and taking the projective limit over all $n$ gives (\ref{conj1forK}).

\section{The Kato conjectures}
We introduce the following notation for Kato complexes:
\begin{definition}\label{definitionkatocomplexes}
\begin{enumerate}
\item For $X$ a scheme over a finite field or the ring of integers in a number field or local field, we denote the complexes
\begin{multline*}
...\r \bigoplus_{x\in X_a}H^{a+1}(k(x),\Z/n(a))\r ...\r 
\bigoplus_{x\in X_{a-1}}H^{a}(k(x),\Z/n(a-1))\r...\\ ...\r 
\bigoplus_{x\in X_1}H^{2}(k(x),\Z/n(1)) \r \bigoplus_{x\in X_0}H^{1}(k(x),\Z/n)
\end{multline*}
by $KC^{(0)}(X,\Z/n\Z)$.
Here the term $\oplus_{x\in X_a}H^{a+1}(k(x),\Z/n(a))$ is placed in degree $a$. We set
$$KH_a^{(0)}(X,\Z/n\Z):= H_a(KC^{(0)}(X,\Z/n\Z)).$$  
The groups $H^{a+1}(k(x),\Z/n(a))$ are the \'etale cohomology groups of $\Spec k(x)$ with coefficients in $\Z/n(a):=\mu^{\otimes a}_n$ if $n$ is invertible on $X$ and $\Z/n(a):=W_r\Omega^{a}_{X_1,\mathrm{log}}[-(a)]\oplus\Z/m(a)$ if $n=mp^r, (m,p)=1,$ is not invertible on $X$ and $X$ is smooth over a field of characteristic $p$.

\item For $X$ a scheme of finite type over a number field $K$ or $K_v, v\in P_K,$ we denote the complexes
\begin{multline*}
...\r \bigoplus_{x\in X_a}H^{a+2}(k(x),\Z/n(a+1))\r ...\r 
\bigoplus_{x\in X_{a-1}}H^{a+1}(k(x),\Z/n(a))\r...\\ ...\r 
\bigoplus_{x\in X_1}H^{3}(k(x),\Z/n(2)) \r \bigoplus_{x\in X_0}H^{2}(k(x),\Z/n(1))
\end{multline*}
by $KC^{(1)}(X,\Z/n\Z)$ and set
$$KH_a^{(1)}(X, \Z/n\Z):= H_a(KC^{(1)}(X,\Z/n\Z))$$.
\item Let $K$ be a global field with ring of integers $\O_K$. Let $U\subset \Spec{O_K}$ be a non-empty open subscheme and $X$ of finite type over $U$. Then there is a natural restriction map
$$KC^{(0)}(X,\Z/n\Z)[1]\r KC^{(1)}(X_K,\Z/n\Z)\r KC^{(1)}(X_{K_v},\Z/n\Z).$$
We define
$$KC^{(0)}(X/U,\Z/n\Z):=\mathrm{cone}[KC^{(0)}(X,\Z/n\Z)[1]\r \bigoplus_{v\in \sum_U} KC^{(1)}(X_{K_v},\Z/n\Z)],$$
where $\sum_U$ denotes the set of places $v\in P_K$ which do not correspond to closed points of $U$. We set  
$$KH^{(0)}_a(X/U,\Z/n\Z):= H_a(KC^{(0)}(X/U,\Z/n\Z)).$$  
\end{enumerate}
\end{definition} 

\begin{remark}
In \cite{Ka86}, Kato constructs the above complexes in greater generality. Let $X$ be an excellent scheme, $n\in \Z-\{0\},q,i\in \Z$ and assume that in the case $q=i+1$, for any prime divisor $p$ of $n$ and for any $x\in X_{0}$ such that char$(k(x))=p$, we have $[k(x):k(x)^p]\geq p^i$. Then there are complexes
\begin{multline*}
C^i_n(X): ...\r \bigoplus_{x\in X_a}H^{i+a+1}(k(x),\Z/n(i+a))\r ...\r 
\bigoplus_{x\in X_1}H^{i+2}(k(x),\Z/n(i+1))\\ \r \bigoplus_{x\in X_0}H^{i+1}(k(x),\Z/n(i))
\end{multline*}
Again the term $\oplus_{x\in X_a}H^{i+a+1}(k(x),\Z/n(i+a))$ is placed in degree $a$ and the homology of $C^i_n(X)$ in degree $a$ is denoted by $KH^{(i)}_a(X,\Z/n\Z)$. 
%The groups $H^{i+a+1}(k(x),\Z/n(i+a))$ are the \'etale cohomology groups of $\Spec k(x)$ with coefficients in $\Z/n(i+a):=\mu^{\otimes i+a}_n$ if $n$ is invertible on $X$ and $\Z/n(i+a):=W_r\Omega^{i+a}_{X_1,\mathrm{log}}[-(i+a)]\oplus\Z/m(i)$ if $n=mp^r, (m,p)=1,$ is not invertible on $X$ and $X$ is smooth over a field of characteristic $p$.

It is shown in \cite{JSS}, that these complexes coincide up to sign with the complexes arising from the appropriate homology theories via the niveau spectral sequence.
\end{remark}

Note that if $X$ is of finite type over the ring of integers $\O_{K_v}$ in a local field $K_v$, then by definition there is an exact triangle $$KC^{(0)}(X,\Z/n\Z)[1]\r KC^{(1)}(X_{K_v},\Z/n\Z)\xrightarrow{\partial} KC^{(0)}(X_v,\Z/n\Z)\r KC^{(0)}(X,\Z/n\Z)$$
which induces an exact sequence of homology groups
\begin{equation}\label{localizationsequencekatohomology}
..\r KH^{(0)}_{a+1}(X,\Z/n\Z)\r KH^{(1)}_{a}(X_{K_v},\Z/n\Z)\r KH^{(0)}_{a}(X_{v},\Z/n\Z)\r KH^{(0)}_{a}(X,\Z/n\Z)\r..
\end{equation}
Let us state Kato's conjectures for the above complexes.
\begin{conj}(\cite[Conj. 0.3]{Ka86})\label{Conj0.3}
Let $X$ be a proper and smooth scheme over a finite field. Then
$$KH_a^{(0)}(X,\Z/n\Z)=0 \quad \mathrm{for} \quad a>0.$$
\end{conj}

\begin{conj}(\cite[Conj. 5.1]{Ka86})\label{Conj5.1}
Let $X$ be a regular scheme proper and flat over $\Spec(\O_K)$, where $\O_K$ is the ring of integers in a local field. Then
$$KH_a^{(0)}(X,\Z/n\Z)=0 \quad \mathrm{for} \quad a\geq 0.$$
\end{conj}

\begin{conj}\label{Conj04}(\cite[Conj. 0.4]{Ka86})
Let $X$ be a proper and smooth scheme over a global field $K$. Then the map
$$KH_a^{(1)}(X,\Z/n\Z)\xrightarrow{\cong} \bigoplus_{v\in P_K} KH_a^{(1)}(X_{K_v},\Z/n\Z)$$
is an isomorphism for $a>0$ and the sequence
$$0\r KH_0^{(1)}(X,\Z/n\Z)\r \bigoplus_{v\in P_K} KH_0^{(1)}(X_{K_v},\Z/n\Z) \r \Z/n\Z\r 0$$
is exact.
\end{conj}

\begin{conj}(\cite[Conj. 0.5]{Ka86})\label{Conj0.5}
Let $\mathcal{X}$ be a regular scheme proper and flat with smooth generic fiber over a non-empty open subscheme $U\subset\Spec(\O_K)$, where $\O_K$ is the ring of integers in a global field. Then
$$KH_a^{(0)}(\mathcal{X}/U,\Z/n\Z)=0 \quad \mathrm{for} \quad a\geq 0.$$
\end{conj}

We add the following conjecture to the list:

\begin{conj}\label{Conj0.6}
Let $\O_S$ be a regular semilocal subring of a global field $K$. Let $\mathcal{S}=\Spec \O_S$.
Let $\mathcal{X}$ be a regular scheme proper and flat over $\mathcal{S}$ with smooth generic fiber $X$. 
Let 
$$KC^{(0)}(\mathcal{X}/S,\Z/n\Z)=\mathrm{coker}[KC^{(0)}(\mathcal{X},\Z/n\Z)[1]\r \bigoplus_{v\in \sum_S} KC^{(1)}(X_{K_v},\Z/n\Z)],$$
where $\sum_S$ denotes the set of places $v\in P_K$ which do not correspond to closed points of $S$.
Then
$$KH_a^{(0)}(\mathcal{X}/S,\Z/n\Z)=0 \quad \mathrm{for} \quad a> 0.$$
\end{conj}

The following is known about these conjectures:
\begin{theorem}\label{ThmKatoconj03}(\cite[Thm. 8.1]{KeS12})
Conjecture \ref{Conj0.3} holds if $n$ is invertible on $X$. If $n$ is not invertible on $X$, then it holds for $a\leq 4$.
\end{theorem}

\begin{theorem}\label{ThmKatoconj51}(\cite[Thm. 8.1]{KeS12})
Conjecture \ref{Conj5.1} holds if $n$ is invertible on $X$.
\end{theorem}

\begin{theorem}\label{ThmKatoconj04}(\cite[Thm. 0.9]{Ja16}, \cite[Thm. 8.3]{KeS12})
Conjecture \ref{Conj04} holds if $n$ is invertible on $X$.
\end{theorem}

\begin{theorem}\label{ThmKatoconj05}(\cite[Thm. 8.4]{KeS12})
Conjecture \ref{Conj0.5} holds if $n$ is invertible on $U$.
\end{theorem}

\begin{theorem}\label{ThmKatoconj06}
Conjecture \ref{Conj0.6} holds if $n$ is invertible on $\mathcal{S}$.
\end{theorem}
\begin{proof}
The proof is analogous to the proof of Theorem \ref{ThmKatoconj05} in \textit{loc.~cit.}. In fact, 
$$KH_a^{(1)}(X,\Z/n\Z)\cong KH_a^{(0)}(\mathcal{X}/S,\Z/n\Z)$$
since by Theorem \ref{ThmKatoconj51} and (\ref{localizationsequencekatohomology}) there is an isomorphism 
$$KH_a^{(1)}(X_{K_v},\Z/n\Z)\xrightarrow{\delta} KH_a^{(0)}(X_v,\Z/n\Z)$$
for $v\in S$.
\end{proof}

%\mortenrem{Configuration complex}

\begin{proposition}\label{propKatoopen}
Let $K$ be a local or global field. Let $n$ be invertible in $K$. If $K$ is a number field assume furthermore that either $n$ is odd or that $K$ has no real places. Let $X$ be a smooth projective scheme over $\Spec(K)$. Let $D\subset X$ be an effective divisor on $X$ and $j:U\hookrightarrow X$ be the inclusion. Let
$$E_1^{p,q}(X,A)=\bigoplus_{x\in X^{(p)}}H_x^{p+q}(X,A)\Rightarrow H^{p+q}(X,A)$$
be  the coniveau spectral sequence (see \ref{Lemmaetaleconiceauspectralseq}).
Then
$$E_1^{\bullet,d+2}(j_!\mu_n^{\otimes d+1})(X)\cong E_1^{\bullet,d+2}(\mu_n^{\otimes d+1})(X).$$
In paritcular for $X$ proper and smooth over a number field the map
$$KH_a^{(1)}(X,j_!\mu_n^{\otimes d+1})\xrightarrow{\cong} \bigoplus_{v\in P_K} KH_a^{(1)}(X_{K_v},j_!\mu_n^{\otimes d+1})$$
is an isomorphism for $a>0$. Here $$KH_a^{(1)}(X,j_!\mu_n^{\otimes d+1}):=E_1^{\bullet,d+2}(j_!\mu_n^{\otimes d+1})(X)$$ and $$KH_a^{(1)}(X_{K_v},j_!\mu_n^{\otimes d+1}):=E_1^{\bullet,d+2}(j_!\mu_n^{\otimes d+1})(X_{K_v}).$$
In the latter case we let $j$ denote the inclusion $D_{K_v}\hookrightarrow X_{K_v}$.
\end{proposition}
\begin{proof}
The proof is identical to the proof of \cite[Prop. 2.2.1]{KeZh}.
\end{proof}

We now turn to finiteness results which can be deduced from the Kato conjectures and which we will need in the following sections.

\begin{lemma}\label{lemmaarchimedeanvanishing}
Let $X$ be of finite type over a Dedekind domain $U\subset \O_K$. Let $v$ be an infinite place of $K$ and $n$ odd. Then the complex
$$KC^{(1)}(X_{K_v},\Z/n\Z)$$
is zero.
\end{lemma}
\begin{proof}
This follows from \ref{TatePoitou}(1). In fact, all the groups $H^{a+1}(k(x),\Z/n(a))$ appearing in the complex $KC^{(1)}(X_{K_v},\Z/n\Z)$ are zero if $n$ is odd.
\end{proof}

\begin{lemma}\label{lemmafiniteness} Let $K_v$ be a non-archimedean local field with residue field $k_v$ of characteristic $p$.
Let $X_{K_v}$ be of finite type over $K_v$. Let $d=\mathrm{dim}X_{K_v}$. Then the following statements hold:
\begin{enumerate}
\item If $n$ is prime to $p$, then the groups $$KH^{(1)}_a(X_{K_v},\Z/n\Z)$$
are finite.
\item If $X_{K_v}$ satisfies condition $(\star)$, then the groups $$KH^{(1)}_a(X_{K_v},\Z/p^n\Z)$$
are finite for $a\leq 2$.
\item If $d = 2$ and $X_{K_v}$ satisfies condition $(\star\star)$, then the groups $$KH^{(1)}_a(X_{K_v},\Z/p^n\Z)$$
are finite for $a \leq 2$ (and hence for all $a$).
\end{enumerate}

\end{lemma}
\begin{proof}
(1) This follows from the exact sequence (\ref{localizationsequencekatohomology}) and Theorems \ref{ThmKatoconj03} and \ref{ThmKatoconj51}.

(2) By Lemma \ref{lemmauzun} there is an exact sequence 
\begin{equation*}
\begin{split}
...\r H_{\text{\'et}}^{2d}(X_{K_v},\Z/p^n\Z(d+1))\r  KH^{(1)}_{2}(X_{K_v},\Z/p^n\Z)\r \CH^{d+1}(X_{K_v},1,\Z/p^n\Z)\\
\xrightarrow{\cong} H_{\text{\'et}}^{2d+1}(X_{K_v},\Z/p^n\Z(d+1)) \r KH^{(1)}_{1}(X_{K_v},\Z/p^n\Z)\r \CH^{d+1}(X_{K_v},0,\Z/p^n\Z)=0
\end{split}
\end{equation*}
The isomorphism is shown in \cite[Thm. 6]{JS12}. The finiteness of the \'etale cohomology groups therefore implies the statement.

(3) Fix a strictly semistable model of $X_{K_v}$ and denote the configuration complex of the special fiber $X_v$ of $X_{K_v}$ by $\Gamma_{X_v}$. Together with the Bloch-Kato conjecture,  \cite[Lemma 7.6]{KeS12} gives a short exact sequence
$$0 \to KH_{a+1}^{(1)}(X_{K_v},\Q_p /\Z_p) / p^n \to KH^{(1)}_a(X_{K_v}, \Z/p^n\Z) \to KH^{(1)}_a(X_{K_v}, \Q_p /\Z_p)[p^n] \to 0$$
where the left group vanishes for $a \geq 2$ by dimension reasons. Furthermore \cite[Thm. 1.4]{JS03} and \cite[Thm. 1.6]{JS03} yield isomorphisms $$KH^{(1)}_a(X_{K_v}, \Q_p /\Z_p) \cong KH^{(0)}_a(X_{v}, \Q_p /\Z_p) \cong H_{a}(\Gamma_{X_v}, \Q_p /\Z_p) $$ which imply the finiteness results as $\Gamma_{X_v}$ is a finite simplicial complex.
\end{proof}

\begin{remark}\label{remarkkatoconj51}
For a regular and projective scheme $X$ over a finite field of characteristic $p$ it is shown in \cite{JS} that the Kato conjecture holds with $\Z/p^r$-coefficients for $a\leq 4$, i.e.
$$KH_a(X,\Z/p^r) = \left\{
\begin{array}{ll}
\Z/p^r & \mathrm{for }\quad a=0 \\
0 & \, \mathrm{otherwise}. \\
\end{array}
\right. $$
In \cite[Sec. C]{JS03}, Jannsen and Saito define a suitable homology theory with $\Z/p^r$-coefficients for a scheme $X$ over a discrete valuation ring $A$ using $p$-adic \'etale Tate twists: 
$$KH_a(X/S,\Z/p^r(-1)):=H^{-a}(X_{et}, Rf^!\Z/p^r(1)_S)$$
with $S=\Spec(A)$ and
$$\Z/p^r(1)_S:= \mathrm{cone}(Rj_*(\Z/p^r(1)_\eta)\r i_*(\Z/p^r)_s[-1])[-1].$$
Unfortunately this theory cannot be used as in the approach of \cite[Sec. 3 (3.11)]{KeS12} to show a Lefschetz theorem implying that $KH_a(X/S,\Z/p^r)=0$ for $a\leq 4$ since there is no appropriate base change and and Artin vanishing for $p$-adic \'etale Tate twists. We are therefore obliged to use Lemma \ref{lemmafiniteness}(2) which is implied by class field theory.
\end{remark}

\section{Main theorem: a local to global theorem}\label{secmainthm}
In this section we prove our our main Theorem \ref{maintheorem2}.
The central method of this article is the comparison of the Zariski and the \'etale motivic cohomology of a regular scheme $X$ over a field or Dedekind domain by analysing the respective coniveau spectral sequences. The difference may in some cases be measured by the Kato conjectures.

\begin{lemma}\label{Lemmaetaleconiceauspectralseq}
Let $X$ be a regular irreducible scheme. Let $A$ be a locally constant constructible sheaf and the stalks be $n$-torsion invertible on $X$. Then the coniveau spectral sequence for \'etale cohomology 
$$E_1^{p,q}(X,A)=\bigoplus_{x\in X^{(p)}}H_x^{p+q}(X,A)\Rightarrow H^{p+q}(X,A)$$
converges and 
$$\bigoplus_{x\in X^{(p)}}H_x^{p+q}(X,A)\cong \bigoplus_{x\in X^{(p)}}H^{q-p}(x,A(-p)).$$ 
\end{lemma}
\begin{proof}
The existence of the spectral sequence is shown in this generality, in fact only assuming that $X$ is equidimensional and noetherian, in \cite{CHK97}. The second statement follows from Gabber's absolute purity theorem (see \cite{Fuj02}) since $X$ is regular.
\end{proof}

\begin{lemma}(\cite[Prop. 2.1]{Ge10})\label{lemmaconiveauchow}
Let $X$ be essentially of finite type over a Dedekind domain $S$. Then the spectral sequence 
$$ ^\CH E_1^{p,q}(X,r)=\bigoplus_{x\in X^{(p)}}\CH^{r-p}(x,-p-q)\Rightarrow \CH^{r}(X,-p-q)$$
converges.
\end{lemma}
\begin{proof}
See \cite[Prop. 2.1]{Ge10}.
\end{proof}

We will need the following general lemma which was originally observed by Jannsen and Saito for schemes over finite fields in \cite{JS}.
\begin{lemma}\label{lemmauzun}
Let $\mathcal{X}$ be a regular scheme of relative dimension $d$ over the spectrum of a field $\mathcal{S}=\Spec(K)$ or a Dedekind scheme $\mathcal{S}=\Spec(\O)$. Let $n$ be invertible on $\mathcal{S}$. For a field $k$ let $cd_n(k)$ denote the $\Z/n\Z$-cohomological dimension of $k$. Then the following statements hold:
\begin{enumerate}
\item (\cite[Thm. 8]{Uz16}) Let $\mathcal{S}=\Spec(K)$ be a field and $X:=\mathcal{X}$. Assume that $c=cd_n(K)-1$. Then the sequence
\begin{equation}
\begin{split}
...\r KH^{(c)}_{q-c+2}(X,\Z/n\Z)\r \CH^{d+c}(X,q,\Z/n\Z)\r H_{\et}^{2(d+c)-q}(X,\Z/n\Z(d+c))\\
\r KH^{(c)}_{q-c+1}(X,\Z/n\Z)\r \CH^{d+c}(X,q-1,\Z/n\Z)\r...
\end{split}
\end{equation}
is  exact.
\item Let $\mathcal{S}=\Spec(\O)$ be of dimension $1$.  Assume that $c=cd_n(K)-2$ and $c=cd_n(k(x))-1$ for $x\in \mathcal{S}_{(0)}$. Then the sequence
\begin{equation}
\begin{split}
...\r KH^{(c)}_{q-c+2}(\mathcal{X},\Z/n\Z)\r \CH^{d+1+c}(\mathcal{X},q,\Z/n\Z)\r H_{\et}^{2(d+1+c)-q}(\mathcal{X},\Z/n\Z(d+1+c))\\
\r KH^{(c)}_{q-c+1}(\mathcal{X},\Z/n\Z)\r \CH^{d+1+c}(\mathcal{X},q-1,\Z/n\Z)\r...
\end{split}
\end{equation}
is  exact.

\end{enumerate}
\end{lemma}
%\begin{proof} \end{proof}

\begin{remark}
The cohomological dimension of a local field is $2$ (see e.g. \cite[Exp. 10, Thm. 2.1]{SGA4}).
The cohomological dimension of a global field $k$ is $2$ in the following cases:
(a) $k$ is a number field and $n$ is odd. (b) $k$ is a number field and does not have any real places. (c) $k$ is a function field of one variable over a finite field and $n$ is prime to $\mathrm{ch}(k)$. This follows from Theorem \ref{TatePoitou}(1). In all of these cases $c=1$.
\end{remark}

We can determine the higher Chow groups of schemes over local fields appearing in the local to global principle of Theorem \ref{maintheorem2} (see also \cite[Sec. 5.2]{Ge10}).

\begin{proposition}
Let $K_v$ be a non-archimedean local field with residue field $k_v$ of characteristic $p$.
Let $X_{K_v}$ be a smooth scheme over $K_v$. Let $d=\mathrm{dim}X_{K_v}$. Let $(n,p)=1$. Then the groups $\CH^{d+1}(X,q,\Z/n\Z)$ are finite for all $q$ and the groups $\CH^{d+1}(X,q,\Z/p^r\Z)$ are finite for $1\geq q$.
\end{proposition}
\begin{proof}
Consider the commutative diagram

$$\begin{xy} 
  \xymatrix@-1.6em{
KH^{(c)}_{q+1}(X_{K_v},\Z/n\Z) \ar[r] \ar[d]^{\cong} & \CH^{d+1}(X_{K_v},q,\Z/n\Z) \ar[r] \ar[d]^{\delta} & H_{\et}^{2d-q+2}(X_{K_v},\Z/n\Z(d+1)) \ar[r] \ar[d]^{\delta} &  KH^{(1)}_{q}(X_{K_v},\Z/n\Z) \ar[d]^{\cong}  \\
KH^{(c)}_{q+1}(X_v,\Z/n\Z)   \ar[r] & \CH^{d}(X_v,q-1,\Z/n\Z)  \ar[r] & H_{\et}^{2d-q}(X_v,\Z/n\Z(d)) \ar[r] & KH^{(c)}_{q}(X_v,\Z/n\Z)
    }
\end{xy} $$
By Lemma \ref{lemmauzun}(1) the rows are exact. The vertical isomorphisms on the left and right of the diagram follow from Theorem \ref{ThmKatoconj51} and the exact sequence (\ref{localizationsequencekatohomology}).
\end{proof}

\begin{figure}
$$\begin{xy} 
  \xymatrix@-1.8em{
 d+2 & H^{d+2}(k(\eta),\Lambda(d+1))\ar[r]^{}   & \bigoplus_{x\in X^{(1)}}H^{d+1}(k(x),\Lambda(d))  \ar[r]^{} & ... \ar[r]^{} & \bigoplus_{x\in X^{(d)}}H^{2}(k(x),\Lambda(1))    \ar[r]^{}  & \\
 d+1 & H^{d+1}(k(\eta),\Lambda(d+1))\ar[r]^{} & \bigoplus_{x\in X^{(1)}}H^{d}(k(x),\Lambda(d))  \ar[r]^{} & ... \ar[r]^{} &  \bigoplus_{x\in X^{(d)}}H^{1}(k(x),\Lambda(1))    \ar[r]^{}  & \\
 ... & & ... &        & &       \\
   1 & & ...&  & & \\
  0 &  &         &   &     \\
   & 0 & 1 & ... & d& 
  }
\end{xy} $$ 
\caption{Table of $E^1_{p,q}(X,\Lambda(j))$ for $X/S$ of relative dimension $d$, $j=d+1$ and $\Lambda:=\Z/n\Z$.}
\label{figure1}
\end{figure}

\begin{theorem}\label{maintheorem2intext}
Let $K$ be a global field, $S$ a set of places of $K$ and $n \in \N_{>1}$. Denote by $\O_S$ the ring of elements in $K$ which are integers at all primes $\mathfrak{p}\notin S$ and let $\mathcal{S}=\Spec \O_S$.
Suppose that
\begin{enumerate}
	\item[(a)] $\mathcal{S}$ is either semi-local or an open of $\Spec(\O_K)$.
	\item[(b)] $n$ is invertible on $\mathcal{S}$.
	\item[(c)] If $K$ is a number field, $S$ contains all archimedian places of $K$ and either $n$ is odd or $K$ has no real places. 
\end{enumerate} 
Let $\mathcal{X}$ be regular, flat and projective of relative dimension $d$ over $\mathcal{S}$ with smooth generic fiber $X$. 
Then the following statements hold:
\begin{enumerate}
\item There is an exact sequence 
$$\CH^{d+1}(\mathcal{X},a,\Z/n\Z)\r  \resprod_{v\in S}\CH^{d+1}(X_{K_v},a,\Z/n\Z) \r H^{a}_{\et}(\mathcal{X}_{},\Z/n\Z)^\vee$$
\item For all $a$ there is a natural surjection
$\Sha(\CH^{d+1}(\mathcal{X},a,\Z/n\Z))\r \Sha^{2d+2-a,d+1}_\et(\mathcal{X}).$
\item If $K$ is a number field and condition ($\star$) holds for $X_{K_v}$ if $v$ divides $n$, then the group $\Sha(\CH^{d+1}(\mathcal{X},1,\Z/n\Z))$ is finite.
\item If $K$ is a function field of one variable over a finite field and $n$ is invertible in $K$, then $\Sha(\CH^{d+1}(\mathcal{X},a,\Z/n\Z))$ is finite for arbitrary $a$.
\end{enumerate}
\end{theorem}

\begin{proof}%[Proof of Theorem \ref{maintheorem2}]
First note that all relevant groups vanish for $a = 0$, so assume $a \geq 1$. If $\dim(\mathcal{S}) = 1$, consider the following commutative diagram:
\[
\begin{tikzcd}
KH^{(0)}_{a+2}(\mathcal{X}/S,\Z/n\Z) \arrow[r, "\cong"]\arrow[d] & \prod_{v\in S} KH^{(1)}_{a+1}(X_{K_v},\Z/n\Z) \arrow[d] \\
\CH^{d+1}(\mathcal{X},a,\Z/n\Z) \arrow[r, "\alpha"] \arrow[d] & \resprod_{v\in S}\CH^{d+1}(X_{K_v},a,\Z/n\Z) \ar[d]  \\
H^{2d+2-a}_{\text{\'et}}(\mathcal{X},\Z/n\Z(d+1)) \arrow[d]  \arrow[r] & \resprod_{v\in S} H^{2d+2-a}_{\text{\'et}}(X_{K_v},\Z/n\Z(d+1)) \arrow[d] \\
KH^{(0)}_{a+1}(\mathcal{X}/S,\Z\/n\Z) \arrow[r, "\cong"] & \prod_{v\in S} KH^{(1)}_{a}(X_{K_v},\Z/n\Z)
\end{tikzcd}
\]
The isomorphisms in the first and the last row follow from Theorem \ref{ThmKatoconj05} and \ref{ThmKatoconj06}. The collumns are exact by Lemma \ref{lemmauzun} and by the assumptions on $K$ and $n$. The third row is part of the generalisation of Saito's Poitou-Tate exact sequence
\[
H^{2d+2-a}_{\text{\'et}}(\mathcal{X},\Z/n\Z(d+1)) \to  \resprod_{v\in S} H^{2d+2-a}_{\text{\'et}}(X_{K_v},\Z/n\Z(d+1)) \to H^a_{\et}(\mathcal{X}, \Z/n\Z)^{\vee} \to \ldots
\]
by Geisser and Schmidt (see \cite{GS18}). The two middle rows also fit into the commutative diagram
\[
\begin{tikzcd}
\Sha(\CH^{d+1}(\mathcal{X},a,\Z/n\Z))\ar[r, hook] \ar[d]_{} & \CH^{d+1}(\mathcal{X},a,\Z/n\Z) \ar[r, "\alpha"] \ar[d]_{}&\resprod_{v\in S}\CH^{d+1}(X_{K_v},a,\Z/n\Z)\arrow[d]\\
\Sha^{2d+2-a,d-1}_\et(\mathcal{X}) \ar[r, hook] & H^{2d+2-a}_{\text{\'et}}(\mathcal{X},\Z/n\Z(d+1)) \ar[r] &\resprod_{v\in S} H^{2d+2-a}_{\text{\'et}}(X_{K_v},\Z/n\Z(d+1))
\end{tikzcd}
\]
The statements $(1)$ and $(2)$ now follow from a diagram chase.

For the case of $S = P_K$, i.e. $\mathcal{S}=\Spec K$ and dim$\mathcal{S}=0$, the proof procedes in the same way, except by starting with the analogous commutative diagram
$$\begin{xy} 
  \xymatrix{
   KH^{(1)}_{a+1}(X,\Z/n\Z) \ar[d]_{} \ar[r]^-{\cong} & \prod_{v\in P_K} KH^{(1)}_{a+1}(X_{K_v},\Z/n\Z)  \ar[d]_{} \\
\CH^{d+1}(X,a,\Z/n\Z) \ar[r]^-{\alpha} \ar[d]_{} & \resprod_{v\in P_K}\CH^{d+1}(X_{K_v},a,\Z/n\Z) \ar[d]_{} \\
H^{2d+2-a}_{\text{\'et}}(X,\Z/n\Z(d+1)) \ar[d]_{}   \ar[r] & \resprod_{v\in P_K} H^{2d+2-a}_{\text{\'et}}(X_{K_v},\Z/n\Z(d+1)) \ar[d]_{}   \\
 KH^{(1)}_{a}(X,\Z/n\Z) \ar[r]^-{\cong} & \prod_{v\in P_K} KH^{(1)}_{a}(X_{K_v},\Z/n\Z).
  }
\end{xy} $$
For $(3)$ and $(4)$ note first that the products $\prod_{v\in P_K}' KH^{(1)}_{a+1}(X_{K_v},\Z/n\Z)$ and $\prod_{v\in P_K}' KH^{(1)}_{a}(X_{K_v},\Z/n\Z)$ are products of finitely many groups since almost all $X_{K_v}$ have good reduction. The statements now follow from Lemma \ref{lemmafiniteness} and the fact that $\Sha_{\et}^{2d+2-a,d-1}(X)$ is finite. The latter follows from \cite{Sa89} (see also \cite[Thm. A]{GS18}).
\end{proof}
\begin{remark}
	Inspecting the diagrams in the proof one finds an isomorphism between the kernel of
	\[
	\Sha(\CH^{d+1}(\mathcal{X},a,\Z/n\Z))\r \Sha^{2d+2-a,d+1}_\et(\mathcal{X})
	\]
	and the cokernel of
	\[
		\resprod_{v\in S}\CH^{d+1}(X_{K_v},a+1,\Z/n\Z) \to  \ker\Big(H^{a+1}_{\et}(\mathcal{X}, \mu_n^{\otimes j})^{\vee} \to H^{2d+2-a}_{\text{\'et}}(X_{K_v},\Z/n\Z(d+1))\Big).
	\]
	One might hence suspect that they fit together into a natural long exact sequence
	\[
		\ldots \to \resprod_{v\in S}\CH^{d+1}(X_{K_v},a+1,\Z/n\Z) \to H^{a+1}_{\et}(\mathcal{X}, \mu_n^{\otimes j})^{\vee}  \to \CH^{d+1}(\mathcal{X},a,\Z/n\Z) \to  \ldots
	\]
	but there seems to be a non-trivial extension problem one needs to solve in oder to prove this. 
\end{remark}
\begin{remark}
For $b\geq d+2, a\geq 0$ and $n$ odd, the map
$$\Sha(\CH^{b}(X,a,\Z/n\Z))\r \Sha^{2b-a,d+1}_\et(X)$$
is an isomorphism and the sequence
$$\CH^{b}(X,a,\Z/n\Z)\r \prod_{v\in P_K}' \CH^{b}(X_{},a,\Z/n\Z)\r H^{2d+2-2b+a}_{\et}(X_{k_v},\mu_n^{\otimes d+1-b})^\vee$$
is exact. This follows from cohomological dimension.
\end{remark}

\begin{corollary}\label{corollarymaintheorem2} 
Let the assumptions be as in Theorem \ref{maintheorem2}. Let $S=P_K$ and $X=\mathcal{X}$. Then the following statements hold:
\begin{enumerate}
\item Conjecture \ref{conj2} holds for $i \geq d+1$ and all $a$.
\item Conjecture \ref{conj3} holds for $i \geq d+1$ and all $a$.
\end{enumerate}
\end{corollary}

\begin{corollary}\label{corCFT}
\begin{enumerate}
\item If $a=1$, then the map
$$\mathrm{C}(X)/n:= \mathrm{coker}[\CH^{d+1}(X,1,\Z/n\Z)\r \prod'_{v\in P_K}\CH^{d+1}(X_{K_v},1,\Z/n\Z)]\r \pi_1^{\mathrm{ab}}(X)/n$$
is an isomorphism.
\item Let $a\geq 2$. The map 
$$\mathrm{coker}[\CH^{d+1}(X,a,\Z/n\Z)\r \prod'_{v\in P_K}\CH^{d+1}(X_{K_v},a,\Z/n\Z)]\r H^{a}_{\et}(X_{},\mu_n^{\otimes j})^\vee$$
is injective.
\end{enumerate}
\end{corollary}
\begin{proof}
The injectivity follows in both cases from Theorem \ref{thm3intext}(3). The surjectivity of $\rho$ in (1) follows from Chebotarev density. A different way to deduce it from the Kato conjectures is the following: consider the diagram with exact rows 
$$\begin{xy} 
  \xymatrix{
  H^1(X,\Z/n\Z)^\vee\ar[r] & H^{2d+2}(X,\mu_n^{d+1})\ar[r] \ar[d]^{\cong} & \prod'H^{2d+2}(X,\mu_n^{d+1})\ar[r] \ar[d]^{\cong} & H^0(X,\Z/n\Z)^\vee \ar[d]^{\cong} & \\
  0 \ar[r] & KH_0^{(1)}(X,\Z/n\Z)  \ar[r] & \bigoplus_{v\in P_K} KH_0^{(1)}(X_{K_v},\Z/n\Z) \ar[r] & \Z/n\Z\ar[r] & 0.
    }
\end{xy} $$
The exactness of the second row follows from Theorem \ref{ThmKatoconj04} and the exactness of the first row from Saito's exact sequence \ref{SaitoTatePoitou}. The vertical maps can be shown to be isomorphisms using the spectral sequence of Lemma \ref{Lemmaetaleconiceauspectralseq} or Lemma \ref{lemmauzun}. The statement now follows from a diagram chase.
\end{proof}

Corollary \ref{corCFT}(1) recovers the unramified class field theory of arithmetic schemes (see \cite{Sa85}). In the next section we strenthen the above argument to also cover the ramified case.

\begin{quest}
$H^{a}_{\et}(X_{},\mu_n^{\otimes j})^\vee$ may be interpreted as a higher homotopy group $"\pi_a(X)/n"$ and should be generated by algebraic cycles. We would like to ask if the map 
$$\mathrm{coker}[\CH^{d+1}(X,a,\Z/n\Z)\r \prod'_{v\in P_K}\CH^{d+1}(X_{K_v},a,\Z/n\Z)]\r H^{a}_{\et}(X_{},\mu_n^{\otimes j})^\vee$$
is surjective, i.e. an isomorphism, and if
$$\Sha(\CH^{d+1}(X,a,\Z/n\Z))=0.$$
\end{quest}

\begin{remark}\label{remarkconj11forcurves}
For $X$ a scheme of finite type over a number field $K$ or $K_v, v\in P_K,$ we denote the complexes
\begin{multline*}
...\r \bigoplus_{x\in X_a}H^{a+1}(k(x),\Z/n\Z(a))\r ...\r 
\bigoplus_{x\in X_{a-1}}H^{a}(k(x),\Z/n\Z(a-1))\r...\\ ...\r 
\bigoplus_{x\in X_1}H^{2}(k(x),\Z/n\Z(1)) \r \bigoplus_{x\in X_0}H^{1}(k(x),\Z/n\Z)
\end{multline*}
by $KC^{(0)}(X,\Z/n\Z)$.
Here the term $\oplus_{x\in X_a}H^{a+1}(k(x),\Z/n\Z(a))$ is placed in degree $a$. We set
$$KH_a^{(0)}(X,\Z/n\Z):= H_a(KC^{(0)}(X,\Z/n\Z)).$$ 

In the case of zero-cycles, i.e. the case $i=d$ of Conjecture \ref{conj1} and Conjecture \ref{conj2}, difference between \'etale and zariski motivic cohomology is measured by $KC^{(0)}(X,\Z/n\Z)$ and one additional row in the coniveau spectral sequence converging to \'etale cohomology. The approach taken in the proof of Theorem \ref{maintheorem2intext} therefore only works in small dimensions.

As mentioned in the introduction, Conjecture \ref{conjA} is known to hold if $X$ is a curve and if the Tate-Shafarevich group of the Jacobian of $X$ does not contain a non-zero element which is infinitely divisible (see \cite[Sec. 7]{Sa89'}, \cite[Sec. 3]{Co99} and \cite[Rem. 1.1(iv)]{Wi12}).
We recall the argument of \cite[Sec. 3]{Co99}, where it is shown that a version of Conjecture \ref{conj2} with $\Q_\ell/\Z_\ell$-coefficients holds for such an $X$, using the above framework. Considering the map of coniveau spectral sequences 
$$ ^\CH E_1^{p,q}(X,d)\r ^{\et}E_1^{p,q}(X,\Z/\ell^n\Z(d))$$
gives the following commutative diagram with exact rows and collums:
$$\begin{xy} 
  \xymatrix{
  0 \ar[d]_{} & 0 \ar[d]_{} & \\
 \CH^{1}(X,\Z/\ell^n\Z) \ar[r]^-{} \ar[d]_{} & \prod_{v\in P_K}\CH^{1}(X_{K_v},\Z/\ell^n\Z) \ar[d]_{} & \\
 H^{2}_{\et}(X,\Z/\ell^n\Z(1)) \ar[d]_{}   \ar[r] & \prod_{v\in P_K} H^{2}_{\et}(X_{K_v},\Z/\ell^n\Z(1)) \ar[d]_{} \ar[r] &  H^{2}_{\et}(X,\mu_{\ell^n}^{\otimes 1})^\vee   \\
  KH_1^{(0)}(X,\Z/\ell^n\Z) \ar[r]^-{} \ar[d]_{} & \bigoplus_{v\in P_K} KH_1^{(0)}(X_{K_v},\Z/\ell^n\Z) \ar[d]_{} & \\
0 & 0 &  &
  }
\end{xy} $$

Now $KH_1^{(0)}(X,\Z/\ell^n\Z)\cong \Br(X)[\ell^n]$ and $KH_1^{(0)}(X,\Z/\ell^n\Z)\cong \Br(X_{K_v})[\ell^n]$. In fact, the collumns also arise from the exact sequence of sheaves $0\r \mu_{\ell^n}\r \O_X^\times\xrightarrow{\cdot\ell^n} \O_X^\times\r 0$. In \cite{Co99}, Colliot-Th\'el\`ene considers a family of cycles $\{z_v\}\in \prod_{v\in P_K}\CH^{1}(X_{K_v})$ which is orthogonal to every class $\xi\in H^{2}_{\et}(X,\Q_\ell/\Z_\ell(1))$ under the pairing 
$$\prod_{v\in P_K}\CH^{1}(X_{K_v},\Z/\ell^n\Z) \times H^{2}_{\et}(X,\Q_\ell/\Z_\ell(1))\r \Q_\ell/\Z_\ell$$
and shows that its image in $\prod_{v\in P_K} H^{2}_{\et}(X_{K_v},\Z/\ell^n\Z(1))$, i.e. after passing to the limit over $n$ in the above diagram, is still in the image of $H^{2}_{\et}(X,\Q_\ell/\Z_\ell(1))$ (see \textit{loc.~cit., Prop. 3.1}). The assertion then follows from the fact that 
$$\mathrm{ker}[\varprojlim_{\ell^n}\Br(X)[\ell^n]\r \varprojlim_{\ell^n}\Br(X_{K_v})[\ell^n]]=0$$
if the Tate-Shafarevich group of the Jacobian of $X$ does not contain a non-zero element which is infinitely divisible (see \textit{loc.~cit., Lem. 3.6}).
\end{remark}

\section{Ramified global class field theory of Kato-Saito}\label{sectionraamifiedcft}
In this section we generalise Corollary \ref{corCFT} to the ramified case. This is very similar to the treatment of class field theory over local and finite fields of Kerz and Zhao in \cite{KeZh}.
\begin{theorem}\label{theoremunramified}
Let $K$ be a global field. Let $n$ be invertible in $K$. If $K$ is a number field assume furthermore that either $n$ is odd or that $K$ has no real places. Let $X$ be a smooth projective scheme over $\Spec(K)$. Let $D\subset X$ be an effective divisor on $X$ and $j:U\hookrightarrow X$ be the inclusion.
Then there is an isomorphism
$$\mathrm{coker}[H^d_\Nis(X,\mathcal{K}^M_{d+1,X|D}/n)\r \resprod_{v\in P_K}H^d_\Nis(X_{K_v},\mathcal{K}^M_{d+1,X_{K_v}|D_{K_v}}/n)]\r \pi_1^{\mathrm{ab}}(U)/n.$$
\end{theorem}
\begin{proof}
There is a commutative diagram with exact rows 
$$\begin{xy} 
  \xymatrix{
      \bigoplus_{x\in X_1} H^{d-1}_x(X_\Nis,\mathcal{K}^M_{d+1,X|D}/n) \ar@{->>}[d]
     \ar[r] & \bigoplus_{x\in X_0} H^d_x(X_\Nis,\mathcal{K}^M_{d+1,X|D}/n) \ar@{->>}[r] \ar[d]^{\cong} & H^d_\Nis(X,\mathcal{K}^M_{d+1,X|D}/n) \ar[d]^{\cong}\\
      \bigoplus_{x\in X_1} H^{2d}_x(X,\mu_n^{d+1})\ar[r]^{\partial}  & \bigoplus_{x\in X_0}H^{2d+1}_x(X,\mu_n^{d+1})\ar[r]  & \mathrm{coker}\partial \\
    }
\end{xy} $$
By \cite{Sa05} the middle vertical map is an isomorphism and the left vertical map is surjective. This implies that the right vertical map is an isomorphism. 
The niveau spectral sequence 
$$E_1^{p,q}(X,j_!\mu_n^{\otimes d+1})=\bigoplus_{x\in X^{(p)}}H_x^{p+q}(X,j_!\mu_n^{\otimes d+1})\Rightarrow H^{p+q}(X,j_!\mu_n^{\otimes d+1}),$$
in which $\mathrm{coker}\partial\cong E_2^{d,d+1}$, and the same argument for the $X_{K_v}$ therefore implies that there is a commutative diagram with exact rows and collums
\[
\begin{tikzcd}
KH^{(1)}_{2}(X,j_!\mu_n^{\otimes d+1}) \arrow[r, "\cong"]\arrow[d] & \prod_{v\in P_K} KH^{(1)}_{2}(X_{K_v},j_!\mu_n^{\otimes d+1}) \arrow[d]& \\
H^d_\Nis(X,\mathcal{K}^M_{d+1,X|D}/n) \arrow[r, "\alpha"] \arrow[d] & \resprod_{v\in P_K}H^d_\Nis(X_{K_v},\mathcal{K}^M_{d+1,X_{K_v}|D_{K_v}}/n) \ar[d] & \\
H^{2d+1}_{\text{\'et}}(X,j_!\mu_n^{\otimes d+1})) \arrow[d]  \arrow[r] & \resprod_{v\in P_K} H^{2d+1}_{\text{\'et}}(X_{K_v},j_!\mu_n^{\otimes d+1})) \arrow[d]\arrow[r]& \pi_1^{\mathrm{ab}}(U)/n  \\
KH^{(1)}_{1}(X,j_!\mu_n^{\otimes d+1}) \arrow[r, "\cong"] & \prod_{v\in P_K} KH^{(1)}_{1}(X_{K_v},j_!\mu_n^{\otimes d+1}).&
\end{tikzcd}
\]
The isomorphisms in the first and last row follow from Proposition \ref{propKatoopen}. The theorem now follows from a diagram chase.

\end{proof}

Passing to the direct limit over all $D$, we obtain a description of $\mathrm{Gal}(\bar{F}/F)^\mathrm{ab}$, where $F$ is the funtion field of $X$. This recovers the following theorem of Kato and Saito:
\begin{theorem}(\cite[Thm. 9.1]{KaS86})
Let $X$ be a projective integral scheme over $\Z$ of dimension $d+1$ with function field $F$. Assume for simplicity that $F$ contains a totally imaginary field if $\mathrm{ch}(F)=0$. Then there is an isomorphism
$$\varprojlim_{n,D}H^{d+1}_\Nis(X,\mathcal{K}^M_{d+1,X|D}/n)\r \mathrm{Gal}(\bar{F}/F)^\mathrm{ab}.$$
\end{theorem}

In fact, it can be shown that taking the inverse limit over the direct sum of the local terms (modulo relations) appearing in the expression of $H^{d+1}_\Nis(X,\mathcal{K}^M_{d+1,X|D}/n)$ via the coniveau spectral sequence, is isomorphic to the cokernel defined using the restricted product in Theorem \ref{theoremunramified}.

\section{The $p$-adic cycle class map}\label{sectionpadic}
In this section let $A$ be a henselian discrete valuation ring of characteristic zero with residue field of characteristic $p$ and function field $K$. Let $X$ be smooth and projective of relative dimension $d$ over $\Spec(A)$. Let $X_n$ denote the thickenings of the special $X_1$. Let $X_K$ denote the generic fiber of $X$.
In \cite[Sec. 10]{KEW16}, Kerz, Esnault and Wittenberg state the following conjecture:
\begin{conj}
Assume the Gersten conjecture for the Milnor K-sheaf $\mathcal{K}^M_{n,X}$. Then the restriction map $$res: \CH^{d}(X)/p^r \r "\mathrm{lim}_n" H^{d}(X_1,\mathcal{K}^M_{d,X_n}/p^r)$$ is an isomorphism.
\end{conj}

The following theorem and its corollary give some evidence for this conjecture.
\begin{theorem}\label{thm3intext}
Let $d=\mathrm{dim}X-1\leq 2$. Then there is an isomorphism
$$\CH^{d+1}(X,1,\Z/p^r\Z)\r H^{2d+1}_\et(X,\mathcal{T}_r(d+1)).$$
\end{theorem}

\begin{proof}%[Proof of Theorem \ref{thm3}]
We use the homology theory defined in Remark \ref{remarkkatoconj51} and consider the associated spectral sequence 
$$E_1^{u,v}=\bigoplus_{x\in X^u}H^{v-u}(x,\Z/p^r\Z(d+1-u))\Rightarrow H^{u+v}_\et(X,\mathcal{T}_r(d+1)).$$
The case $d=1$ is clear. Let us therefore assume that $d=2$. Then 
$$E_2^{2,3}=\CH^{d+1}(X,1,\Z/p^r\Z).$$
We therefore need to show that $E_2^{\bullet,4}=0$. By \cite{JSS}, the complex $E_1^{\bullet,4}$ coincides up to sign with the relevant complex defined by Kato.
Consider the localization sequence (\ref{localizationsequencekatohomology}) 
$$KH^{(0)}_{a+1}(X_{1},\Z/p^r\Z)\r KH^{(0)}_{a+1}(X,\Z/p^r\Z)\r KH^{(1)}_{a}(X_{K},\Z/p^r\Z) $$
for the homology theory.
The group $KH^{(0)}_{a+1}(X_{1},\Z/p^r\Z)$ vanishes for $a\leq 3$ by the results on the Kato conjectures with $p$-coefficients in low degrees cited in Remark \ref{remarkkatoconj51}. For the vanishing of the group $ KH^{(1)}_{d}(X_{K},\Z/p^r\Z)$ consider the short exact sequence
$$0 \to KH_{a+1}^{(1)}(X_{K_v},\Q_p /\Z_p) / p^n \to KH^{(1)}_a(X_{K_v}, \Z/p^n\Z) \to KH^{(1)}_a(X_{K_v}, \Q_p /\Z_p)[p^n] \to 0$$
(\cite[Lemma 7.6]{KeS12}). Now $KH^{(1)}_3(X_{K_v}, \Q_p /\Z_p))$ vanishes for dimension reasons and for $a\leq 2$ we have that 
$ KH^{(1)}_a(X_{K_v}, \Q_p /\Z_p)\cong KH^{(1)}_a(X_{v}, \Q_p /\Z_p)=0$
by \cite[Thm. 1.6]{JS03}. 
\end{proof}

Note that due to the twist by $d+1$ we have full purity for the logarithmic deRham-Witt sheaves. This seems to be the main difference to the zero-cycle case.

\section{A finiteness theorem for arithmmetic schemes}\label{sectionfiniteness}
A version of Bass' finiteness conjecture predicts that for a regular scheme $X$ of finite type over $\Z$, the groups
$$\CH^r(X,q)$$
are finitely generated. This conjecture is known to hold for $r=1$ or dim$(X)=1$ by results of Quillen. In arbitrary dimension $d$ there are few results. If $X$ proper and flat over $\Spec\Z$, then $\CH^{d}(X)$ is finite (see \cite{Sa85}) by unramified class field theory. If $X$ smooth projective over a finite field, then the groups $\CH^{d+j}(X,j)$ are finite for $j\geq 0$ by unramified class field theory (see \cite{CSS83} and \cite{KaS86} for $j=0$ and \cite[Sec. 6]{Lu17} for $j\geq 1$).  

As an application of the Kato conjectures, Kerz and Saito show that for any quasi-projective scheme $X$ of dimension $d$ over a finite field $k$ and $n$ invertible on $X$,  the groups 
$$\CH^d(X,q,\Z/n\Z)$$
are finite for all $q\geq 0$ (see \cite[Cor. 9.4]{KeS12}). In the following theorem we establish some small-degree cases for arithmetic schemes (see also \cite[Sec. 7.2]{Ge10}):

\begin{theorem}\label{theoremfinitenessintext}
	Let $K$ be a global field with ring of integers $\O_K$, let $U \subset \Spec \O_K$ be open, nonempty and let $n \in \N_{>1}$ be invertible on $U$.
	Let $X$ be a regular connected scheme, proper and flat over $U$ with smooth generic fiber $X_K$ and let $d = \dim(X) -1 = \dim(X)$.
	If $K$ is a number field assume furthermore that either $n$ is odd or that $K$ has no real places. 
	\begin{enumerate}
		\item Suppose that for all places $v$ of $K$ dividing $n$, $X_{K_v}$ satisfies $(\star)$.		
		Then the groups
		$$\CH^{d+1}(X,a,\Z/n\Z)$$
		are finite for all $1\geq a\geq 0$.
		\item Suppose $d = 2$ and for all places $v$ of $K$ dividing $n$, $X_{K_v}$ satisfies $(\star\star)$. Then the groups $$\CH^{d+1}(X,a,\Z/n\Z)$$ are finite for all $a \geq 0$.
	\end{enumerate}
		
\end{theorem}

\begin{proof} By Lemma \ref{lemmauzun}(2) there is an exact sequence
	\begin{equation}\label{JSexactsequence}
	\begin{split}
	...\r KH^{(0)}_{q+2}(X,\Z/n\Z)\r \CH^{d+1}(X,q,\Z/n\Z)\r H_{\text{\'et}}^{2d+2-q}(X,\Z/n\Z(d+1))\\
	\r KH^{(0)}_{q+1}(X,\Z/n\Z)\r \CH^{d+1}(X,q-1,\Z/n\Z)\r...
	\end{split}
	\end{equation}
	By \cite[Ch. II, 7.1]{Mi86}, the \'etale cohomology groups $H_{\text{\'et}}^{2d+2-q}(X,\Z/n\Z(d))$ are known to be finite for $n$ invertible on $U$. It therefore suffices to show the finiteness of $KH^{(0)}_{q+2}(X,\Z/n\Z)$. By definition we have an exact sequence 
	$$...\r KH^{(0)}_{a}(X/U,\Z/n\Z)\r KH^{(0)}_{a}(X,\Z/n\Z)\r \bigoplus_{v\in \sum_U}KH^{(1)}_{a-1}(X_{K_v},\Z/n\Z))\r ..$$
	By Theorem \ref{ThmKatoconj05}, $KH^{(0)}_{a}(X/U,\Z/n\Z)$ vanishes for $a>0$ and is isomorphic to $\Z/n\Z$ for $a=0$. The statement now follows from Lemma \ref{lemmafiniteness}.
\end{proof}

\bibliographystyle{plain}
\bibliography{Bibliografie}

\end{document}